\documentclass{amsart}
\usepackage{amssymb}
\usepackage{tikz}
\usepackage{cite}
\usepackage{float}
\restylefloat{figure}

 \newtheorem{thm}{Theorem}[section]

 \newtheorem{prop}[thm]{Proposition}
 \theoremstyle{definition}
 
 \theoremstyle{remark}

 \numberwithin{equation}{section}

\newcommand{\lz}[1]{\langle #1 \rangle}
\newcommand{\rk}[1]{\text{rk}\left(#1\right)}
\newcommand{\Z}{\mathbb{Z}}
\newcommand{\R}{\mathbb{R}}
\newcommand{\C}{\mathbb{C}}
\newcommand{\Q}{\mathbb{Q}}
\newcommand{\Ha}{\mathbb{H}}
\newcommand{\RP}{\mathbb{R}P}
\newcommand{\CP}{\C P}
\newcommand{\del}{\partial}

\begin{document}

\title[Reflection groups of the quadratic form $-px_0^2+x_1^2+\ldots+x_n^2$]
 {Reflection groups of the quadratic form \\$-px_0^2+x_1^2+\ldots+x_n^2$ with $p$ prime.}

\author[Mark]{Alice Mark}

\begin{abstract}
We present the classification of reflective quadratic forms $-px_0^2+x_1^2+\ldots+x_n^2$ for $p$ prime.  We show that for $p = 5$, it is reflective for $2\leq n\leq 8$, for $p = 7\text{ and }17$ it is reflective for $n = 2\text{ and }3$, for $p=11$ it is reflective for $p=2,3,\text{ and } 4$, and it is not reflective for higher values of $n$.  We also show that it is non-reflective for $n > 2$ when $p = 13,19,\text{and }23$.  This completes the classification of these forms with $p$ prime. 
\end{abstract}

\maketitle

\section{Introduction}

Arithmetic hyperbolic reflection groups are contained in the groups of units of certain integral quadratic forms, and we study them by studying those forms.  A quadratic form is said reflective if its units group is generated up to finite index by finitely many reflections.  There are only finitely many reflective arithmetic quadratic forms.  This was proved separately and by different methods in two papers from around the same time by Nikluin \cite{nik2} and by Agol, Belolipetsky, Storm, and Whyte \cite{agol2006finiteness}.  In \cite{nik}, Nikulin lists all the strongly square-free lattices of rank 3 over $\Z$.   In \cite{allcock} Allcock  extended that list to include all of the rank 3 lattices over $\Z$, including the non-strongly-square-free ones.  One of the tools used by both Nikulin and Allcock in their classifications is Vinberg's algorithm, which is a method for producing the simple roots of a lattice given the associated quadratic form.  Vinberg \cite{vin} and much later McLeod \cite{mc} studied quadratic forms
\begin{equation}
\label{form}
f(x)=-px_0^2+x_1^2+\ldots+x_n^2
\end{equation}
with $p = 1,2,3$.  In the present paper we consider these same forms for other prime values of $p$.  These same results and other related ones can be found in McLeod's Ph.D. thesis from 2013 \cite{mcleodthesis}.   The results we present here are a subset of McLeod's, though our argument for why the first non-reflective case is non-reflective is slightly different from his.

 We restrict to $p$ prime because it simplifies our implementation of Vinberg's algorithm.  We wish to know for which values of $p$ prime and $n\geq 2$ this form is reflective.  Nikulin's list \cite{nik} contains no lattices of prime determinant larger than 23.  All reflective lattices with form \eqref{form}, $n=2$, and square-free determinant appear on his list.   We show that when $p = 5$, \eqref{form} is reflective for $2\leq n\leq 8$, when $p = 7\text{ and }17$ it is reflective for $n = 2\text{ and }3$, and when $p = 11$ it is reflective for $n=2,3,\text{ and } 4$.  We show that it is non-reflective in all higher dimensions than these.

The $p = 17$ and $p=23$ cases have an interesting feature.  For all other values of $p$, the first non-reflective case fails to be reflective because the fundamental polyhedron contains a facet that passes out of hyperbolic space at a point that cannot be a vertex.  This is not what happens when $p = 17$ and when $p=23$.  In both of these cases, all the vertices of all the facets are inside hyperbolic space.  We show that there are infinitely many facets by exhibiting an infinite order symmetry of the polyhedron.  Indeed, in both cases the symmetry group of the fundamental polyhedron is infinite dihedral.  We will give some details about this polyhedron when $p=23$ and $n=3$ because it has an interesting shape, and being $3$-dimensional it is possible to visualize.

We implemented our version of Vinberg's algorithm in C++ with the PARI Library.  We verified the computations by hand for $p=5,7,\text{ and }11$.  For higher values of $p$ this was impracitical.

In what follows, let $V=V^{n+1}$ be an $(n+1)$-dimensional real vector space with basis $v_0,\ldots,v_n$, and quadratic form \eqref{form}.  Let $L = L^{n+1}$ be the integer lattice generated by the same basis.  The group $\Theta$ of integral automorphisms is the group of symmetries of $L$ preserving the form \eqref{form} and mapping each connected component of the set $\{x:f(x)<0\}$ to itself.  This group splits as a semidirect product 
$$\Theta=\Gamma\rtimes H$$
where $\Gamma$ is generated by reflections and $H$ is a group of symmetries of an associated polyhedron in hyperbolic $n$-space \cite{vin}.  We say $L$ is reflective if $H$ is a finite group. 

Vinberg gives an algorithm for finding the simple roots of such a lattice in \cite{vin}.  By \emph{root}, we mean a primitive vector $r=\sum_{i=0}^nk_iv_i\in L$ of positive norm such that reflection across the mirror $r^{\perp}$ preserves the lattice.  Roots are characterized by the crystallographic condition, which reduces to 
\begin{equation}
\label{crystal}
\frac{2k_i}{(r,r)}\in\Z\text{ for } i>0 \text{ and } \frac{2pk_0}{(r,r)}\in\Z
\end{equation}
The inner product here is given by the quadratic form \eqref{form}. This leads to a nice simplification in our case, since having prime determinant means there are fewer possible norms for roots.  A system of simple roots is a subset of all the roots that spans the whole space, and with the property that any root is a non-negative linear combination of the simple roots. The inner product between any two simple roots is never positive.  The mirrors of the simple roots point outward from the walls of one copy of the fundamental polyhedron for the reflection group acting on hyperbolic space.

Vinberg's algorithm performs a batch search in order of increasing height, where height is given by the formula 
\begin{equation}
\frac{k_0^2}{(r,r)}
\end{equation}
and is understood to be the distance of a potential new root from the control vector $v_0$.  We extend a system of simple roots $e_1,\ldots e_n$ for the stabilizer of $v_0$, where $e_i = v_{i+1}-v_i$ for $1\leq i < n$ and $e_n = -v_n$.

\section{The reflective lattices}

 We will treat the $p=5$ case in detail to show how the computation works.  For the other cases we include Coxeter diagrams and tables of roots in the appendix.

When $p=5$, \eqref{crystal} implies that if $r$ is a root, $(r,r) = 1,2,5,$ or $10$, and if $(r,r)=5$ or $10$ then $5\nmid k_0$ and $5|k_j$ for $j\neq 0$.

\begin{table}
\begin{tabular}{clccc}
\hline
$\frac{k_0^2}{(e_i,e_i)}$&$e_i$&$(e_i,e_i)$&$i$&$n$\\
\hline\hline
$\frac{1}{2}$&$v_0+2v_1+v_2+v_3+v_4$&$2$&$n+3$&$\geq 4$\\
&$v_0+v_1+v_2+v_3+v_4+v_5+v_6+v_7$&$2$&$n+4$&$\geq 7$\\\\
$\frac{4}{5}$&$2v_0+5v_1$&$5$&$n+1$&$\geq 2$\\\\
$\frac{1}{1}$&$v_0+2v_1+v_2+v_3$&$1$&$n+3$&$3$\\
&$v_0+v_1+v_2+v_3+v_4+v_5+v_6$&$1$&$n+4$&$6$\\\\
$\frac{9}{5}$&$3v_0+5v_1+5v_2$&$5$&$n+2$&$\geq 2$\\
\hline
\end{tabular}

\caption{Vectors found with Vinberg's Algorithm when $p = 5$.  The labels $i$ are chosen for convenience to later arguments rather than the order in which the algorithm finds them.}
\label{table1}
\end{table}

\begin{prop} The first several vectors that Vinberg's algorithm produces are listed in Table \ref{table1}.
\end{prop}

\begin{proof}The batches labeled $\frac{1}{10}$,$\frac{1}{5}$, and $\frac{4}{10}$ are empty because $(e,e)=5$ or $10$, and there is no way to write $(e,e)+5k_0^2=15$, $20$ or $30$ as a sum of squares of integers all divisible by $5$.

The batch labeled $\frac{1}{2}$ consists of vectors $e=\sum_{i=0}^nk_nv_n$ where 
$$\sum_{i=1}^nk_i^2=(e,e)+5k_0^2=7$$
There are two ways to write $7$ as a sum of squares.  This batch contains one vector if $n\geq 4$, and two if $n\geq 7$.  These two vectors have inner product $0$, so they correspond to orthogonal walls of the fundamental chamber.

The batch labeled $\frac{4}{5}$ consists of vectors $e=\sum_{i=0}^nk_iv_i$ where
$$\sum_{i=1}^nk_i^2=(e,e)+5k_0^2=25$$
and $5|k_i$ for all $i>0$.  The vector $2v_0+5v_1$ has inner product $0$ and $-5$ with the vectors in the previous nonempty batch so we keep it for all $n\geq 2$.

The batch labeled $\frac{9}{10}$ is empty since $(e,e)=10$ and $(e,e)+5k_0^2=55$ cannot be written as a sum of squares of integers all divisible by $5$.

The batch labeled $\frac{1}{1}$ consists of vectors $e=\sum_{i=0}^nk_iv_i$ where
$$\sum_{i=1}^nk_i^2=(e,e)+5k_0^2=6$$
There are two ways to write $6$ as a sum of squares.  One of these produces the vector $v_0+2v_1+v_2+v_3$.  This has inner product $0$ with the vector in batch $\frac{4}{5}$, so we keep it when $n=3$, but it has positive inner product with the vector $v_0+2v_1+v_2+v_3+v_4$ in batch $\frac{1}{2}$, so we throw it out when $n\geq 4$.

The other way to write $6$ as a sum of squares produces the vector $v_0+v_2+v_3+v_4+v_5+v_6$.  This has inner product $0$ with the vector $v_0+2v_1+v_2+v_3+v_4$ in batch $\frac{1}{2}$ and inner product $-5$ with the vector in batch $\frac{4}{5}$, so we keep it when $n=6$.  It has positive inner product with the vector $v_0+v_1+v_2+v_3+v_4+v_5+v_6+v_7$ in batch $\frac{1}{2}$, so we throw it out when $n\geq 7$.

The batch labeled $\frac{16}{10}$ is empty since $(e,e)=10$ and $(e,e)+5k_0^2=90$ cannot be written as a sum of squares of integers all divisible by $5$.

The batch labeled $\frac{9}{5}$ consists of vectors $e=\sum_{i=0}^nk_iv_i$ where
$$\sum_{i=1}^nk_i^2=(e,e)+5k_0^2=50$$
and $5|k_i$ for all $i>0$.  There is exactly one vector satisfying this.  $3v_0+5v_1+5v_2$ has non-positive inner product with all the vectors in previous batches, so we keep it for all $n\geq 2$.

\end{proof}

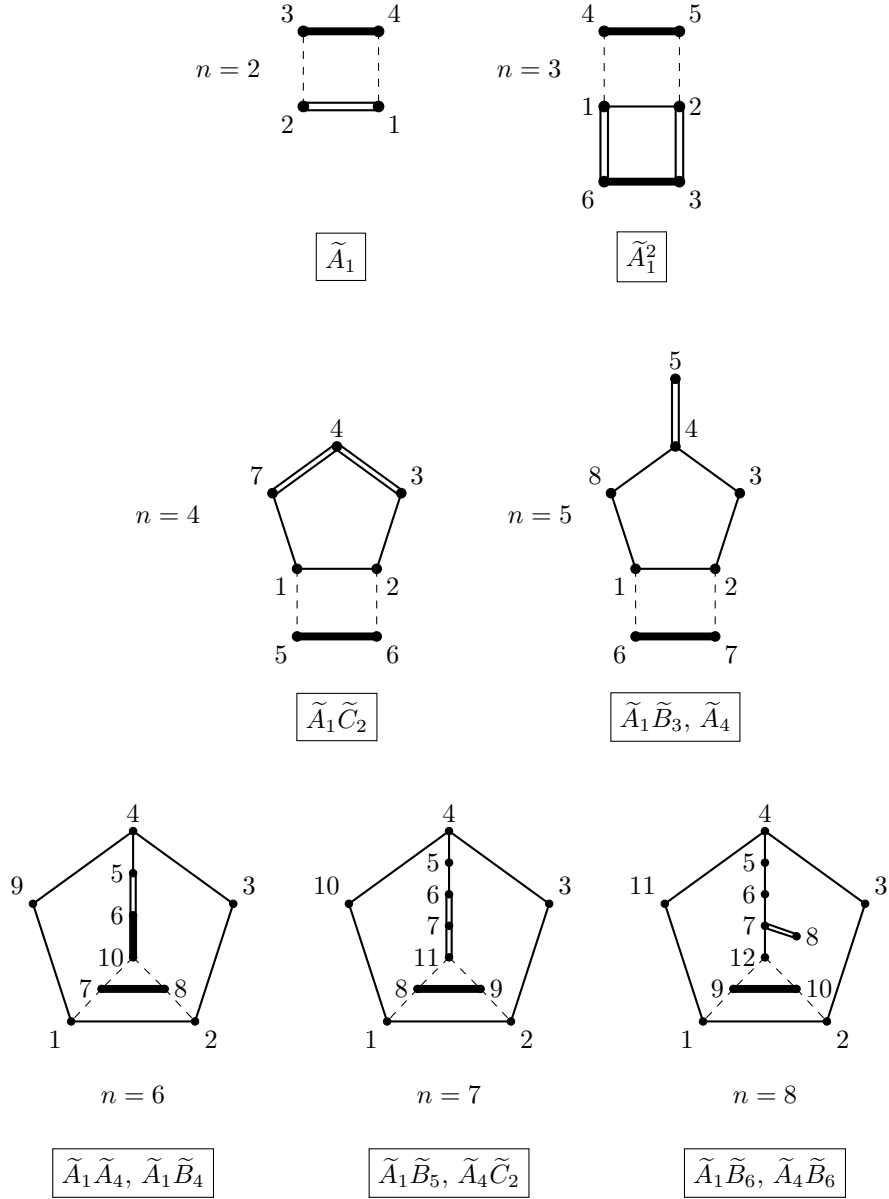
\begin{figure}[ht]
\begin{center}
\begin{tikzpicture}[scale=1]
\node[draw=none, fill=none] at (0,-3){};

\node[draw=none, fill=none] at (-3,.5) {$n=2$};
\node[draw, align = center] at (-1.5,-2) {$\widetilde{A}_1$};     
\filldraw (-2,1) circle (2pt) node[anchor = south east]{3};
\draw[thick,line width = 3] (-2,1) -- (-1,1);
\filldraw (-1,1) circle (2pt) node[anchor = south west]{4};
\draw[dashed, - ] (-1,1) -- (-1,0);
\filldraw (-1,0) circle (2pt) node[anchor = north west]{1};
\draw[thick,-] (-1,.05) -- (-2,.05);
\draw[thick,-] (-1,-.05) -- (-2,-.05);
\filldraw (-2,0) circle (2pt) node[anchor = north east]{2};
\draw[dashed,-] (-2,0) -- (-2,1);
    
\node[draw=none, fill=none] at (1,.5) {$n=3$};
\node[draw, align = center] at (2.5,-2) {$\widetilde{A}_1^2$};
\filldraw (2,1) circle (2pt) node[anchor = south east]{4};
\draw[thick,line width = 3] (2,1) -- (3,1);
\draw[dashed,-] (2,1) -- (2,0);
\filldraw (3,1) circle (2pt) node[anchor = south west]{5};
\draw[dashed,-] (3,1) -- (3,0);
\filldraw (2,0) circle (2pt) node[anchor = east]{1};
\draw[thick,-] (2,0) -- (3,0);
\draw[thick,-] (1.95,0) -- (1.95,-1);
\draw[thick,-] (2.05,0) -- (2.05,-1);
\filldraw (3,0) circle (2pt) node[anchor = west]{2};
\draw[thick,-] (2.95,0) -- (2.95,-1);
\draw[thick,-] (3.05,0) -- (3.05,-1);
\filldraw (2,-1) circle (2pt) node[anchor = north east]{6};
\draw[thick,line width = 3] (2,-1) -- (3,-1);
\filldraw (3,-1) circle (2pt) node[anchor = north west]{3};

\end{tikzpicture}

\begin{tikzpicture}[scale=.9]
\node[draw=none, fill=none] at (0,-4){};

\node[draw=none, fill=none] at (-2.5,0) {$n=4$};
\node[draw, align = center] at (0,-3) {$\widetilde{A}_1\widetilde{C}_2$};    
\filldraw (0,1) circle (2pt) node[anchor = south]{4};
\draw[thick,-] (.03,1.03) -- (.981,.339);
\draw[thick,-] (-.03,.97) -- (.921,.279);
\draw[thick,-] (-.03,1.03) -- (-.981,.339);
\draw[thick,-] (.03,.97) -- (-.921,.279);
\filldraw (.951,.309) circle (2pt) node[anchor = south west]{3};
\draw[thick,-] (.951,.309) -- (.588,-.809);
\filldraw (-.951,.309) circle (2pt) node[anchor = south east]{7};
\draw[thick,-] (-.951,.309) -- (-.588,-.809);
\filldraw (.588,-.809) circle (2pt) node[anchor = north west]{2};
\draw[thick,-] (.588,-.809) -- (-.588,-.809);
\draw[dashed,-] (.588,-.809) -- (.588,-1.809);
\filldraw (-.588,-.809) circle (2pt) node[anchor = north east]{1};
\draw[dashed,-] (-.588,-.809) -- (-.588,-1.809);
\filldraw (.588,-1.809) circle (2pt) node[anchor = north west]{6};
\draw[thick, line width = 3] (.588,-1.809) -- (-.588,-1.809);
\filldraw (-.588,-1.809) circle (2pt) node[anchor = north east]{5};    
    
\node[draw=none, fill=none] at (3,0) {$n=5$};
\node[draw, align = center] at (5,-3) {$\widetilde{A}_1\widetilde{B}_3$, $\widetilde{A}_4$};
\filldraw (5,2) circle (2pt) node[anchor = south]{5};
\draw[thick,-] (4.95,2) -- (4.95,1);
\draw[thick,-] (5.05,2) -- (5.05,1);
\filldraw (5,1) circle (2pt) node[anchor = south west]{4};
\draw[thick,-] (5,1) -- (5.951,.309);
\draw[thick,-] (5,1) -- (4.049,.309);
\filldraw (5.951,.309) circle (2pt) node[anchor = south west]{3};
\draw[thick,-] (5.951,.309) -- (5.588,-.809);
\filldraw (4.049,.309) circle (2pt) node[anchor = south east]{8};
\draw[thick,-] (4.049,.309) -- (4.412,-.809);
\filldraw (5.588,-.809) circle (2pt) node[anchor = north west]{2};
\draw[thick,-] (5.588,-.809) -- (4.412,-.809);
\draw[dashed,-] (5.588,-.809) -- (5.588,-1.809);
\filldraw (4.412,-.809) circle (2pt) node[anchor = north east]{1};
\draw[dashed,-] (4.412,-.809) -- (4.412,-1.809);
\filldraw (5.588,-1.809) circle (2pt) node[anchor = north west]{7};
\draw[thick, line width = 3] (5.588,-1.809) -- (4.412,-1.809);
\filldraw (4.412,-1.809) circle (2pt) node[anchor = north east]{6};    
    
\end{tikzpicture}

\begin{tikzpicture}[scale=.7]

\node[draw=none, fill=none] at (0,-3) {$n=6$};
\node[draw, align = center] at (0,-4.5) {$\widetilde{A}_1\widetilde{A}_4$, $\widetilde{A}_1\widetilde{B}_4$};
\filldraw (0,2) circle (2pt) node[anchor = south]{4};
\draw[thick,-] (0,2) -- (1.902,.618);
\draw[thick,-] (0,2) -- (-1.902,.618);
\filldraw (1.902,.618) circle (2pt) node[anchor = south west]{3};
\draw[thick,-] (1.902,.618) -- (1.176,-1.618);
\filldraw (-1.902,.618) circle (2pt) node[anchor = south east]{9};
\draw[thick,-] (-1.902,.618) -- (-1.176,-1.618);
\filldraw (1.176,-1.618) circle (2pt) node[anchor = north west]{2};
\draw[thick,-] (1.176,-1.618) -- (-1.176,-1.618);
\filldraw (-1.176,-1.618) circle (2pt) node[anchor = north east]{1};
\draw[thick,-] (0,2) -- (0,1.2);
\filldraw (0,1.2) circle (2pt) node[anchor = east]{5};
\draw[thick,-] (-.05,1.2) -- (-.05,.4);
\draw[thick,-] (.05,1.2) -- (.05,.4);
\filldraw (0,.4) circle (2pt) node[anchor = east]{6};
\draw[thick, line width = 3] (0,.4) -- (0,-.4);
\filldraw (0,-.4) circle (2pt) node[anchor = east]{10};
\draw[dashed,-] (0,-.4) -- (-.6,-1);
\draw[dashed,-] (0,-.4) -- (.6,-1);
\filldraw (-.6,-1) circle (2pt) node[anchor = east]{7};
\draw[thick, line width = 3] (-.6,-1) -- (.6,-1);
\draw[dashed,-] (-.6,-1) -- (-1.176,-1.618);
\filldraw (.6,-1) circle (2pt) node[anchor = west]{8};
\draw[dashed,-] (.6,-1) -- (1.176,-1.618);
    
\node[draw=none, fill=none] at (6,-3) {$n=7$};
\node[draw, align = center] at (6,-4.5) {$\widetilde{A}_1\widetilde{B}_5$, $\widetilde{A}_4\widetilde{C}_2$};
\filldraw (6,2) circle (2pt) node[anchor = south]{4};
\draw[thick,-] (6,2) -- (7.902,.618);
\draw[thick,-] (6,2) -- (4.098,.618);
\filldraw (7.902,.618) circle (2pt) node[anchor = south west]{3};
\draw[thick,-] (7.902,.618) -- (7.176,-1.618);
\filldraw (4.098,.618) circle (2pt) node[anchor = south east]{10};
\draw[thick,-] (4.098,.618) -- (4.824,-1.618);
\filldraw (7.176,-1.618) circle (2pt) node[anchor = north west]{2};
\draw[thick,-] (7.176,-1.618) -- (4.824,-1.618);
\filldraw (4.824,-1.618) circle (2pt) node[anchor = north east]{1};
\draw[thick,-] (6,2) -- (6,1.4);
\filldraw (6,1.4) circle (2pt) node[anchor = east]{5};
\draw[thick,-] (6,1.4) -- (6,.8);
\filldraw (6,.8) circle (2pt) node[anchor = east]{6};
\draw[thick,-] (5.95,.8) -- (5.95,.2);
\draw[thick,-] (6.05,.8) -- (6.05,.2);
\filldraw (6,.2) circle (2pt) node[anchor = east]{7};
\draw[thick,-] (5.95,.2) -- (5.95,-.4);
\draw[thick,-] (6.05,.2) -- (6.05,-.4);
\filldraw (6,-.4) circle (2pt) node[anchor = east]{11};
\draw[dashed,-] (6,-.4) -- (5.4,-1);
\draw[dashed,-] (6,-.4) -- (6.6,-1);
\filldraw (5.4,-1) circle (2pt) node[anchor = east]{8};
\draw[thick, line width = 3] (5.4,-1) -- (6.6,-1);
\draw[dashed,-] (5.4,-1) -- (4.824,-1.618);
\filldraw (6.6,-1) circle (2pt) node[anchor = west]{9};
\draw[dashed,-] (6.6,-1) -- (7.176,-1.618);

\node[draw=none, fill=none] at (12,-3) {$n=8$};
\node[draw, align = center] at (12,-4.5) {$\widetilde{A}_1\widetilde{B}_6$, $\widetilde{A}_4\widetilde{B}_6$};
\filldraw (12,2) circle (2pt) node[anchor = south]{4};
\draw[thick,-] (12,2) -- (13.902,.618);
\draw[thick,-] (12,2) -- (10.098,.618);
\filldraw (13.902,.618) circle (2pt) node[anchor = south west]{3};
\draw[thick,-] (13.902,.618) -- (13.176,-1.618);
\filldraw (10.098,.618) circle (2pt) node[anchor = south east]{11};
\draw[thick,-] (10.098,.618) -- (10.824,-1.618);
\filldraw (13.176,-1.618) circle (2pt) node[anchor = north west]{2};
\draw[thick,-] (13.176,-1.618) -- (10.824,-1.618);
\filldraw (10.824,-1.618) circle (2pt) node[anchor = north east]{1};
\draw[thick,-] (12,2) -- (12,1.4);
\filldraw (12,1.4) circle (2pt) node[anchor = east]{5};
\draw[thick,-] (12,1.4) -- (12,.8);
\filldraw (12,.8) circle (2pt) node[anchor = east]{6};
\draw[thick,-] (12,.8) -- (12,.2);
\filldraw (12,.2) circle (2pt) node[anchor = east]{7};
\draw[thick,-] (12,.2) -- (12,-.4);
\draw[thick,-] (12.06,.23) -- (12.63,.03);
\draw[thick,-] (11.97,.17) -- (12.57,-.03);
\filldraw (12.6,0) circle (2pt) node[anchor = west]{8};
\filldraw (12,-.4) circle (2pt) node[anchor = east]{12};
\draw[dashed,-] (12,-.4) -- (11.4,-1);
\draw[dashed,-] (12,-.4) -- (12.6,-1);
\filldraw (11.4,-1) circle (2pt) node[anchor = east]{9};
\draw[thick, line width = 3] (11.4,-1) -- (12.6,-1);
\draw[dashed,-] (11.4,-1) -- (10.824,-1.618);
\filldraw (12.6,-1) circle (2pt) node[anchor = west]{10};
\draw[dashed,-] (12.6,-1) -- (13.176,-1.618);
    
\end{tikzpicture}
\end{center}

\caption{Coxeter diagrams for the hyperbolic reflection groups associated to the form $-5x_0^2+x_1^2+\ldots+x_n^2$, and their maximal affine subdiagram types.}
\label{fig1}
\end{figure}

\begin{prop}
The diagrams in figure \ref{fig1} all describe acute angled polyhdra of finite volume.
\end{prop}

To prove this we use Vinberg's criteron for finite volume, which is Proposition 1 in \cite{vin}.   We restate the criterion here for convenience.
\begin{prop}
The necessary and sufficient condition for the polyhedron $P$ with gram matrix $S$ to have finite volume is that if $G_S$ is a critical principal submatrix of $G$, 
\begin{enumerate}
\item[(a)] if $G_S$ is positive semi-definite, then there is a superset $T$ of $S$ such that $G_T$ is positive semi-definite of rank $n-1$.
\item[(b)] if $G_S$ is indefinite and nondegenerate, then $K_S = \{0\}$.
\end{enumerate}
\end{prop}

Here, $G_S$ is the principal submatrix of $G$ corresponding to some subset $S$ of the index set $I$ of the roots.  Notice that if $G_S$ is critical it must be either positive semi-definite or indefinite and nondegenerate.  The polyhedral angle $K = \{x\in V:(x,e_i)\leq 0 \text{ for all } i\}$ is the continuation of the polyhedron $P$.  $K_s= \{x\in K:(x,e_i)= 0 \text{ for all } i\in S\}$ is the subset of $K$ that is fixed by reflections in the roots whose indices are in $S$.

If conditions (a) and (b) of the finite volume criterion are met, then any facet of the polyhedral angle $K$ that passes out of the cone $\mathfrak{C} = \{x\in V:f(x)<0\}$ intersects the boundary of $\mathfrak{C}$ in a line corresponding to a vertex at infinity of $P$.

In the proof we refer also to the list of the affine Coxeter diagrams which can also be found in Table 2 of \cite{vin}.

\begin{proof}

The only cocompact hyperbolic diagrams occurring as subdiagrams of the graphs in figure \ref{fig1} are dotted line edges, and the only maximal affine subdiagrams have rank $n-1$. 

Listed below each diagram in Figure \ref{fig1}
the types of maximal affine subdiagrams that diagram has.  Every affine subdiagram that appears can be extended to one of these, which shows that condition (a) of the finite volume criterion is satisfied for all of the diagrams in Figure \ref{fig1}.

It remains to show that condition (b) of the finite volume criterion is met.  There are $4$ different dotted line components which appear.  We will show that each satisfies condition (b).

A sufficient condition for a dotted line subgraph to satisfy condition (b) is that there must exist a subset $T$ of the vertex set such that the diagram with vertex set $T$ is a spherical subdiagram of rank at least $n-1$, and there are no edges between the vertices in $S$ and the vertices in $T$ (this is a corollary of Proposition 2 in \cite{vin}).

The two dotted line subgraphs that appear only when $n=6,7,8$ can be shown to satisfy condition (b) using this fact.  These are the two dotted line subgraphs that include the vertex $n+4$, and by symmetry we need to only go through the argument for one of them.  Let $S=\{n+1,n+4\}$, and let $T=\{2,3,\ldots,n-1,n+3\}$.  When $n=6$, $T$ has type $D_5$.  When $n=7$, $T$ has type $E_6$.  When $n=8$, $T$ has type $E_7$.  In each case this is a spherical diagram of rank $n-1$ with no edges joining any vertex in $T$ to any vertex in $S$.

The remaining two dotted line subgraphs appear for all $n$, and checking them requires a more explicit calculation.  The first of these has vertex set $S=\{1,n+1\}$.  The two associated roots are
\[e_1 = -v_1+v_2 \text{ ~ }\text{ ~ }\text{ ~ and ~ }\text{ ~ }\text{ ~ } e_{n+1}=2v_0+5v_1\]
so a vector $v\in K_S$ (fixed by reflections with respect to both $e_1$ and $e_{n+1}$) has the form 
$$v=av_0+2a(v_1+v_2)+\sum_{i=3}^nk_iv_i$$
Since $K_S\subset K$, $v\in K_S$ must satisfy $(v,e_i)\leq 0 $ for all $e_i$.  In particular, this holds for $i= 1,\ldots,n$ so we have $2a\geq k_3\geq\ldots\geq k_n\geq 0$, so $a\geq 0$.  We also have that $(v,e_n+2) = -15a+20a\leq 0$, so $a\leq 0$.  We conclude that $a=0$, and therefore $K_S=\{0\}$.

By symmetry, we need not repeat the argument for  $S=\{2, n+2\}$.

\end{proof}

For $p = 7\text{ and }17$, \eqref{form} is reflective for $n = 2\text{ and }3$.  For $p = 11$ it is reflective for $n = 2,3,4$.  For $p = 13,19,\text{ and } 23$, \eqref{form} is reflective for $n=2$.  Tables of vectors found with Vinberg's algorithm are in the appendix.  For $p=7,11,\text{ and }17$, the Coxeter diagrams are also in the appendix.  Since the only reflective case for other values of $p$ is $n=2$,  we describe the fundamental polygon by its norm/angle sequence and its symmetries instead of by a Coxeter diagram.  A norm/angle sequence is a symbol that describes a 2-dimensional hyperbolic polygon.  It has the form 
$${a_1}_{n_1}{a_2}_{n_2}\ldots {a_k}_{n_k}$$
where $a_i$ is the norm of the outward pointing root orthogonal to the $i$-th side, and $\frac{\pi}{n_i}$ is the angle between the $i$-th side and the $(i+1)$-st side.  Cyclic permutations of such a symbol describe the same polygon if we take $a_{k+1} = a_1$.  If a polygon has a rotation symmetry, we write
$$\left({a_1}_{n_1}{a_2}_{n_2}\ldots {a_k}_{n_k}\right)^m$$
to indicate that the whole symbol can be obtained by concatenating $m$ copies of the partial symbol.  These polygon symmetries need not preserve the form \eqref{form}, but in all of our cases they do.

\section{Non-reflective lattices in higher dimensions}

The following proposition ensures that once we show that \eqref{form} is nonreflective for some $n$, we don't need to check any higher dimensions.

\begin{prop}
If \eqref{form} is non-reflective for some $n \geq 2$, then is also non-reflective in all higher dimensions.
\end{prop}
\begin{proof}
The proof is by induction.  Consider the root $e_{n+1} = -v_{n+1}$ in $L^{n+1}$.  Its mirror contains a face of the fundamental polyhedron $P$.  Because $e_{n+1}^2 = 1$, all the faces of $P$ incident with $e_{n+1}^{\perp}$ meet it in an angle of $\frac{\pi}{even}$.  Thus for each such face $F$, there is a root $r$ of $L^{n+1}$ that is orthogonal to $e_{n+1}$ whose mirror intersects $e_{n+1}^{\perp}$ in the same line as $F$.  The reflections in these roots preserve $e_{n+1}^{\perp}$ and $L^{n+1}$, so they are roots of the lattice $e_{n+1}^{\perp}\cap L^{n+1}$.

The lattice $e_{n+1}^{\perp}\cap L^{n+1}$ is just the complement to $v_{n+1}$, which is to say the lattice generated by $v_o\ldots,v_{n}$.  So it is exactly $L^n$, with quadratic form \eqref{form}.  By assumption this lattice is non-reflective, so it has infinitely many simple roots.  Thus $L^{n+1}$ also has infinitely many simple roots.  Thus by induction \eqref{form} is non-reflective in all higher dimensions.
\end{proof}

\begin{prop}
When $p = 5$ and $n\geq 9$, \eqref{form} is non-reflective.
\end{prop}

\begin{figure}[t]

\begin{center}
\begin{tikzpicture}[scale=1]

  \node[draw=none, fill=none] at (0,-2.5) {$n = 9$};    
  \filldraw (0,2) circle (2pt) node[anchor = south]{4};
  \draw[thick,-] (0,2) -- (1.902,.618);
  \draw[thick,-] (0,2) -- (-1.902,.618);
  \filldraw (1.902,.618) circle (2pt) node[anchor = south west]{3};
  \draw[thick,-] (1.902,.618) -- (1.176,-1.618);
  \filldraw (-1.902,.618) circle (2pt) node[anchor = south east]{12};
  \draw[thick,-] (-1.902,.618) -- (-1.176,-1.618);
  \filldraw (1.176,-1.618) circle (2pt) node[anchor = north west]{2};
  \draw[thick,-] (1.176,-1.618) -- (-1.176,-1.618);
  \filldraw (-1.176,-1.618) circle (2pt) node[anchor = north east]{1};
  \draw[thick,-] (0,2) -- (0,1.4);
  \filldraw (0,1.4) circle (2pt) node[anchor = east]{5};
  \draw[thick,-] (0,1.4) -- (0,.8);
  \filldraw (0,.8) circle (2pt) node[anchor = east]{6};
  \draw[thick,-] (0,.8) -- (0,.2);
  \filldraw (0,.2) circle (2pt) node[anchor = east]{7};
  \draw[thick,-] (0,.2) -- (0,-.4);
  \draw[thick,-] (0,.2) -- (.4,0);
  \filldraw (0.4,0) circle (2pt) node[anchor = south west]{8};
  \draw[thick,-] (.43,.03) -- (.83,-.17);
  \draw[thick,-] (.37,-.03) -- (.77,-.23);
  \filldraw (.8,-.2) circle (2pt) node[anchor = south west]{9};
  \filldraw (0,-.4) circle (2pt) node[anchor = east]{13};
  \draw[dashed,-] (0,-.4) -- (-.6,-1);
  \draw[dashed,-] (0,-.4) -- (0.6,-1);
  \filldraw (-.6,-1) circle (2pt) node[anchor = east]{10};
  \draw[thick, line width = 3] (-.6,-1) -- (0.6,-1);
  \draw[dashed,-] (-.6,-1) -- (-1.176,-1.618);
  \filldraw (0.6,-1) circle (2pt) node[anchor = west]{11};
  \draw[dashed,-] (0.6,-1) -- (1.176,-1.618);
    
\end{tikzpicture}	
\end{center}
\caption{Coxeter diagram of the polyhedron in dimension $n = 9$ with the first four vectors found by Vinberg's algorithm}
\label{fig2}
\end{figure}
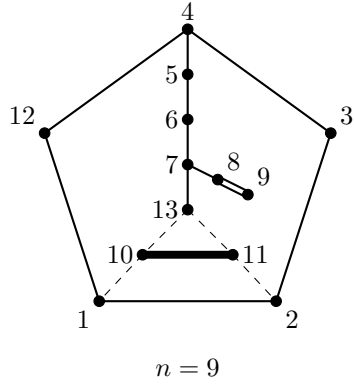

\begin{proof}
By the previous proposition, we only need to show that it is non-reflective when $n=9$.  Vinberg's algorithm finds the same first 4 vectors as when $n=8$.  The Coxeter diagram obtained by adding these four roots is shown in Figure \ref{fig2}.  

The diagram has a subdiagram of type $\widetilde{D}_7$ labeled in Figure \ref{fig2} by the numbers $3,4,5,6,7,8,12,13$.  This affine subdiagram corresponds to a primitive vector $e\in V$ of norm $0$, obtained by taking a certain positive linear combination of the corresponding roots: 
$$e = e_3+2e_4+2e_5+2e_6+2e_7+e_8+e_{12}+e_{13} = 2v_0+3v_1+2v_2+v_3+\ldots+v_9$$
The faces of $K = \{x\in V:(x,e_i)\leq 0 \text{ for all } i\}$ corresponding to roots making up this copy of $\widetilde{D}_7$ pass through $e$, so for $P$ to have finite volume, $e$ must correspond to a vertex of $P$ at infinity.

Let $M = e^{\perp}\cap L$, and $\overline{M} = M/\lz{e}$.  Lemma 3.1 in \cite{vin3} says that $e$ corresponds to a vertex of $P$ at infinity if and only if the positive definite lattice $\overline{M}$ is spanned by roots. $\overline{M}$ contains a primitive $D_7$ sublattice spanned by the $\overline{e}_3,\ldots,\overline{e}_8,\overline{e}_{12}$, where the bar denotes the image in the quotient by $\lz{e}$.  The orthogonal complement to this sublattice in $\overline{M}$ is generated by the image in the quotient of a primitive vector $f_1=v_0+5v_2 $ of norm $20$.  This vector is not a root since its norm is not $1,2,5$, or $10$.  The sum of the $D_7$ with its complement has index $4$ in $\overline{M}$ with cyclic quotient, so the glue vector $f_2 = v_2-2v_3$ is needed to generate all of $\overline {M}$.  The vector $f_2$ has norm $5$, but is also not a root since the coefficients of $v_i$ for $i>0$ are not all divisible by $5$.  Since  $\overline{M}$ is not generated by roots, $e$ does not correspond to a vertex of $P$ at $\infty$, and $P$ does not have finite volume.
\end{proof}

For $p = 7,11,13,\text{ and }19$, the proof is essentially the same.  We find a primitive vector of norm zero corresponding to a rank $n-2$ affine subdiagram, and show that its orthogonal sublattice is not generated by roots.  The tables listing all these vectors are in the appendix.  For $p = 17$, $n=4$ and $p = 23$, $n=3$, Vinberg's algorithm will never produce an affine subdiagram of rank $n-2$, so we need to do something slightly different.  For both of these, we show that the polyhedron is infinite by showing it has a symmetry of infinite order.

When $p=23$ and $n=3$, it is nice to try to visualize the polyhedron since it is 3-dimensional.  It is an infinite cylinder tiled by hexagonal faces.  Four faces fit around the cylinder.  Its symmetry group is generated by an order 2 rotation $R$ fixing the center of two opposite faces, and a glide reflection $G$ that anti-commutes with $R$.  Figure \ref{23poly} shows an unwrapped picture of a piece of its surface.

The first 14 vectors found by Vinberg's algorithm are listed in Table \ref{23table} and are labeled on the diagram in Figure \ref{23poly}.  The glide reflection $G$ takes the corner  $\{e_1,e_{10},e_7\}$ to the corner $\{e_8,e_6,e_4\}$.  The rotation $R$ takes the corner $\{e_1,e_{10},e_7\}$ to the corner $\{e_1,e_3,e_2\}$.  One can write down the matrices for $R$ and $G$, and check that they preserve the form \eqref{form}, and that they generate the infinite dihedral group
$$\mathcal{D} = \lz{G,R:R^2, GRGR}$$
By looking at the angles between the facets, one can see that $\mathcal{D}$ acts transitively on the various type of corners and that it is the full symmetry group of the polyhedron.

We can look for a generalized lattice Weyl vector by looking for eigenvectors of infinite order elements of $\mathcal{D}$ with eigenvalue $1$.  The index $4$ subgroup generated by $G^2$ fixes a vector of positive norm.  The root system has hyperbolic type (note that our sign convention is the opposite of Nikulin's \cite{nik}).

Similarly in the $p=17$ case we find that the symmetry group of the 4-dimensional hyperbolic polyhedron is an infinite dihedral group, and there is a generalized lattice Weyl vector of positive norm.  Thus we have the following:

\begin{prop}
The form \eqref{form} is nonreflective when $p = 23$ and $n = 3$, and when $p=17$ and $n=4$.
\end{prop}

Even though it is difficult to visualize a 4-dimensional hyperbolic polyhedron in any sort of concrete way, we can get an idea of what the automorphism group is doing by looking at how it acts on the Coxeter diagram.  We included so many diagrams in the appendix because they are so useful for visualizing polyhedron symmetries.

\pagebreak
\appendix

\section{Tables and diagrams}
\subsection{Tables and diagrams for $p=7$}

\begin{table}[H]
\begin{tabular}{clccc}
\hline
$\frac{k_0^2}{(e_i,e_i)}$&$e_i$&$(e_i,e_i)$&$i$&$n$\\
\hline\hline
$\frac{1}{2}$&$v_0+3v_1$&$2$&$n+1$&$\geq 2$\\
&$v_0+2v_1+2v_2+v_3$&$2$&$n+2$&$\geq 3$\\\\
$\frac{1}{1}$&$v_0+2v_1+2v_2$&$1$&$n+2$&$2$\\
\hline
\end{tabular}
\caption{Vectors found with Vinberg's Algorithm for $p = 7$}
\label{table7}
\end{table}

\begin{figure}[H]
\begin{center}
\begin{tikzpicture}[scale=1]
  \node[draw=none, fill=none] at (-1.5,-1.5) {$n=2$};    
  \filldraw (-2,.2) circle (2pt) node[anchor = south east]{4};
  \draw[thick,-] (-2,.25) -- (-1,.25);
  \draw[thick,-] (-2,.15) -- (-1,.15);
  \filldraw (-1,.2) circle (2pt) node[anchor = south west]{3};
  \draw[dashed, - ] (-1,.2) -- (-1,-.8);
  \filldraw (-1,-.8) circle (2pt) node[anchor = north west]{1};
  \draw[thick,-] (-1,-.75) -- (-2,-.75);
  \draw[thick,-] (-1,-.85) -- (-2,-.85);
  \filldraw (-2,-.8) circle (2pt) node[anchor = north east]{2};
  \draw[dashed,-] (-2,-.8) -- (-2,.2);

  \node[draw=none, fill=none] at (2,-1.5) {$n=3$};    
  \filldraw (2,1) circle (2pt) node[anchor = south]{3};
  \draw[thick,-] (2.03,1.03) -- (2.981,.339);
  \draw[thick,-] (1.97,.97) -- (2.921,.279);
  \draw[thick,-] (1.97,1.03) -- (1.019,.339);
  \draw[thick,-] (2.03,.97) -- (1.079,.279);
  \filldraw (2.951,.309) circle (2pt) node[anchor = south west]{5};
  \draw[thick,-] (2.951,.309) -- (2.588,-.809);
  \draw[thick,-](2.951,.309) -- (1.049,.309);
  \filldraw (1.049,.309) circle (2pt) node[anchor = south east]{2};
  \draw[thick,-] (1.049,.309) -- (1.412,-.809);
  \filldraw (2.588,-.809) circle (2pt) node[anchor = north west]{4};
  \draw[dashed,-] (2.588,-.809) -- (1.412,-.809);
  \filldraw (1.412,-.809) circle (2pt) node[anchor = north east]{1};    
\end{tikzpicture}
\end{center}
\caption{$p=7$, reflective with $n = 2\text{ and }3$}
\end{figure}
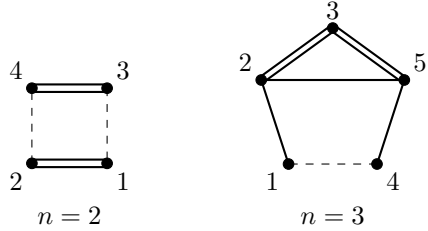

\begin{figure}[H]
\begin{center}
\begin{tikzpicture}[scale=1]
  \filldraw (0,2) circle (2pt) node[anchor = south]{4};
  \draw[thick,-] (-.05,2) -- (-.05,1);
  \draw[thick,-] (.05,2) -- (.05,1);
  \filldraw (0,1) circle (2pt) node[anchor = south west]{3};
  \draw[thick,-] (0,1) -- (0.951,.309);
  \draw[thick,-] (0,1) -- (-.951,.309);
  \filldraw (0.951,.309) circle (2pt) node[anchor = south west]{6};
  \draw[thick,-] (0.951,.309) -- (0.588,-.809);
  \draw[thick,-] (.951,.309) -- (-.951,.309);
  \filldraw (-.951,.309) circle (2pt) node[anchor = south east]{2};
  \draw[thick,-] (-.951,.309) -- (-.588,-.809);
  \filldraw (.588,-.809) circle (2pt) node[anchor = north west]{5};
  \draw[dashed,-] (0.588,-.809) -- (-.588,-.809);
  \filldraw (-.588,-.809) circle (2pt) node[anchor = north east]{1};  
\end{tikzpicture}	
\end{center}
\caption{The first non-reflective case with $p = 7$ is $n = 4$}
\begin{tabular}{ll}
Affine subdiagram: & $\widetilde{A}_2$\\ 
Norm 0 vector: & $v_0+2v_1+v_2+v_3+v_4$\\
Complement: & $2v_0+7v_1$\\
Complement norm: & $21$\\
Order 3 glue: & $v_1-2v_2$\\
\end{tabular}
\end{figure}

\pagebreak

\subsection{Tables and diagrams for $p=11$}
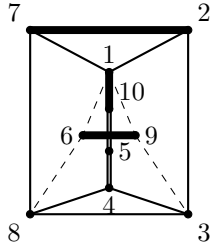
\begin{table}[H]
\begin{tabular}{clccc}
\hline
$\frac{k_0^2}{(e_i,e_i)}$ & $e_i$ & $(e_i,e_i)$ & $i$ & $n$\\
\hline\hline
$\frac{9}{22}$ & $3v_0+11v_1$ & $22$ & $n+1$ & $\geq 2$\\\\
$\frac{1}{2}$ & $v_0+3v_1+2v_2$ & $2$ & $n+2$ & $\geq 2$\\
& $v_0+2v_1+2v_2+2v_3+v_4$ & $2$ & $n+3$ & $\geq 4$\\
& $3v_0+v_1+v_2+v_3+v_4+v_5$ & $2$ & $n+5$ & $\geq 5$\\\\
$\frac{1}{1}$ & $v_0+2v_1+2v_2+2v_3$ & $1$ & $n+3$ & $3$\\
& $v_0+3v_1+v_2+v_3+v_4$ & $1$ & $n+5$ & $\geq 4$\\\\
$\frac{64}{22}$ & $8v_0+22v_1+11v_2+11v_3$ & $22$ & $n+4$ & $\geq 3$\\ 
\hline
\end{tabular}
\caption{Vectors found with Vinberg's Algorithm for $p = 11$}
\label{table11}
\end{table}

\begin{figure}[H]
\begin{center}
\begin{tikzpicture}[scale=.7]
  \node[draw=none, fill=none] at (-2,-1.5) {$n=2$};
  \filldraw (-3,1) circle (2pt) node[anchor = south east]{1};
  \draw[thick,-] (-3.05,1) -- (-3.05,-1);
  \draw[thick,-] (-2.95,1) -- (-2.95,-1);
  \draw[dashed,-] (-3,1) -- (-1,1);
  \filldraw (-3,-1) circle (2pt) node[anchor = north east]{2};
  \draw[dashed,-] (-3,-1) -- (-1,-1);
  \filldraw (-1,-1) circle (2pt) node[anchor = north west]{3};
  \draw[thick,-] (-1,-1) -- (-3,1);
  \filldraw (-1,1) circle (2pt) node[anchor = south west]{4};
  
  \node[draw=none, fill=none] at (2,-1.5) {$n=3$};
  \filldraw (1,1) circle (2pt) node[anchor = south east]{4};
  \draw[thick,line width = 3] (1,1) -- (3,1);
  \draw[dashed,-] (1,1) -- (1,-1);
  \draw[dashed,-] (1,1) -- (2,.3);
  \filldraw (3,1) circle (2pt) node[anchor = south west]{7};
  \draw[dashed,-] (3,1) -- (2,.3);
  \draw[dashed,-] (3,1) -- (3,-1);
  \filldraw (2,.3) circle (2pt) node[anchor = south]{1};
  \draw[thick,-] (2,.3) -- (1.8,-.5);
  \draw[thick,-] (2,.3) -- (2.2,-.5);
  \filldraw (1.8,-.5) circle (2pt) node[anchor = south east]{5};
  \draw[thick, line width = 3] (1.8,-.5) -- (2.2,-.5);
  \draw[thick,-] (1.83,-.53) -- (1.03,-1.03);
  \draw[thick,-] (1.77,-.47) -- (.97,-.97);
  \filldraw (2.2,-.5) circle (2pt) node[anchor = south west]{2};
  \draw[thick,-] (2.23,-.47) -- (3.03,-.97);
  \draw[thick,-] (2.17,-.53) -- (2.97,-1.03);
  \filldraw (1,-1) circle (2pt) node[anchor = north east]{6};
  \draw[dashed,-] (1,-1) -- (3,-1);
  \filldraw (3,-1) circle (2pt) node[anchor = north west]{3};

  \node[draw=none, fill=none] at (6,-2) {$n=4$};
  \filldraw (5,1.5) circle (2pt) node[anchor = south east]{6};
  \draw[thick,line width = 3] (5,1.5) -- (7,1.5);
  \draw[thick,-] (5,1.5) -- (5,-1.5);
  \draw[thick,-] (5,1.5) -- (6,.7);
  \filldraw (7,1.5) circle (2pt) node[anchor = south west]{2};
  \draw[thick,-] (7,1.5) -- (6,.7);
  \draw[thick,-] (7,1.5) -- (7,-1.5);
  \filldraw (6,.7) circle (2pt) node[anchor = south]{1};
  \draw[dashed,-] (6,.7) -- (5.8,0);
  \draw[dashed,-] (6,.7) -- (6.2,0);
  \draw[dashed,-] (6,.7) -- (6,-.5);
  \filldraw (5.8,0) circle (2pt) node[anchor = south east]{5};
  \draw[thick, line width = 3] (5.8,0) -- (6.2,0);
  \draw[dashed,-] (5.8,0) -- (5,-1.5);
  \filldraw (6.2,0) circle (2pt) node[anchor = south west]{9};
  \draw [dashed,-] (6.2,0) -- (7,-1.5);
  \filldraw(6,-.5) circle (2pt) node[anchor = west]{8};
  \draw[thick, line width = 3] (6,-.5) -- (6,-1);
  \filldraw(6,-1) circle (2pt) node[anchor = north]{4};
  \draw[thick,-] (6.03,-1.03) -- (5.03,-1.53);
  \draw[thick,-] (5.97,-.97) -- (4.97,-1.47);
  \draw[thick,-] (6.03,-.97) -- (7.03,-1.47);
  \draw[thick,-] (5.95,-1.03) -- (6.97,-1.53);
  \filldraw (5,-1.5) circle (2pt) node[anchor = north east]{7};
  \draw[thick,-] (5,-1.5) -- (7,-1.5);
  \filldraw (7,-1.5) circle (2pt) node[anchor = north west]{3};

\end{tikzpicture}
\end{center}
\caption{$p=11$, reflective with $n=2, 3, \text{ and }4$}
\end{figure}
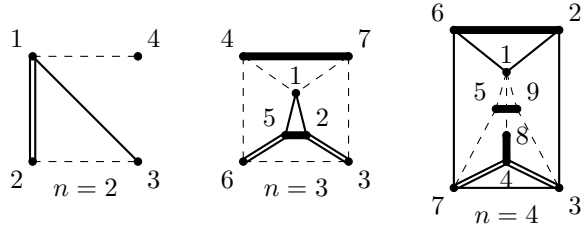

\begin{figure}[H]
\begin{center}
\begin{tikzpicture}[scale=.7]

  \filldraw (4,1.5) circle (2pt) node[anchor = south east]{7};
  \draw[thick,line width = 3] (4,1.5) -- (7,1.5);
  \draw[thick,-] (4,1.5) -- (4,-2);
  \draw[thick,-] (4,1.5) -- (5.5,.7);
  \filldraw (7,1.5) circle (2pt) node[anchor = south west]{2};
  \draw[thick,-] (7,1.5) -- (5.5,.7);
  \draw[thick,-] (7,1.5) -- (7,-2);
  \filldraw (5.5,.7) circle (2pt) node[anchor = south]{1};
  \draw[dashed,-] (5.5,.7) -- (5,-.5);
  \draw[dashed,-] (5.5,.7) -- (6,-.5);
  \draw[thick,line width = 3] (5.5,.7) -- (5.5,0);
  \filldraw(5.5,0) circle (2pt) node[anchor = south west]{10};
  \draw[thick, -] (5.47,0) -- (5.47,-.8);
  \draw[thick, -] (5.53,0) -- (5.53,-.8);
  \filldraw (5,-.5) circle (2pt) node[anchor = east]{6};
  \draw[thick, line width = 3] (5,-.5) -- (6,-.5);
  \draw[dashed,-] (5,-.5) -- (4,-2);
  \filldraw (6,-.5) circle (2pt) node[anchor = west]{9};
  \draw [dashed,-] (6,-.5) -- (7,-2);
  \filldraw(5.5,-.8) circle (2pt) node[anchor = west]{5};
  \draw[thick,-] (5.47,-.8) -- (5.47,-1.5);
  \draw[thick,-] (5.53,-.8) -- (5.53,-1.5);
  \filldraw (5.5,-1.5) circle (2pt) node[anchor = north]{4};
  \draw[thick,-] (5.5,-1.5) -- (4,-2);
  \draw[thick,-] (5.5,-1.5) -- (7,-2);
  \filldraw (4,-2) circle (2pt) node[anchor = north east]{8};
  \draw[thick,-] (4,-2) -- (7,-2);
  \filldraw (7,-2) circle (2pt) node[anchor = north west]{3};
\end{tikzpicture}
\end{center}
\caption{The first non-reflective case with $p=11$ is $n = 4$}
\begin{tabular}{ll}
Affine subdiagram: & $\widetilde{A}_1\widetilde{A}_2$\\ 
Norm 0 vector: & $v_0+2v_1+2v_2+v_3+v_4+v_5$\\
Complement: & $4v_0+11v_1+11v_2$\\
Complement norm: & $66$\\
Order 6 glue: & $v_1-2v_2$\\
\end{tabular}
\end{figure}

\pagebreak

\subsection{Tables and diagrams for $p=13$}

\begin{table}[H]
\begin{tabular}{clccc}
\hline
$\frac{k_0^2}{(e_i,e_i)}$ & $e_i$ & $(e_i,e_i)$ & $i$ & $n$\\
\hline\hline
$\frac{1}{1}$ & $v_0+3v_1+2v_2+v_3$ & $1$ & $4$ & $3$\\\\
$\frac{25}{13}$ & $5v_0+13v_1+13v_2$ & $13$ & $3$ & $2$\\\\
$\frac{4}{1}$ & $2v_0+7v_1+2v_2$ & $1$ & $4$ & $2$\\\\
$\frac{64}{13}$ & $8v_0+26v_1+13v_2$ & $13$ & $5$ & $2$\\\\
$\frac{324}{13}$ & $18v_0+65v_1$ & $13$ & $6$ & $2$\\\\
$\frac{144}{2}$  & $12v_0+43v_1+5v_2$ & $2$ & $7$ & $2$\\\\
$\frac{2209}{13}$ & $47v_0+169v_1+13v_2$ & $13$ & $8$ & $2$\\
\hline
\end{tabular}
\caption{Vectors found with Vinberg's Algorithm for $p = 13$}

\begin{tabular}{l}
Norm/angle sequence for $n=2$ is  $\left(2_41_213_{\infty}13_2\right)^2$.\\
Polygon has an order 2 rotation preserving \eqref{form}.
\end{tabular}
\label{table13}
\end{table}

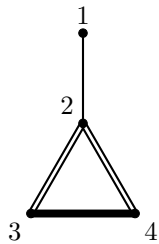
\begin{figure}[H]
\begin{center}
\begin{tikzpicture}[scale=.8]
    \filldraw (0,2.5) circle (2pt) node[anchor = south]{1};
    \draw[thick,-] (0,2.5) -- (0,1);
    \filldraw (0,1) circle (2pt) node[anchor = south east]{2};
    \draw[thick,-] (-.03,1.03) -- (-.896,-.47);
    \draw[thick,-] (.03,.97) -- (-.836,-.53);
    \draw[thick,-] (-.03,.97) -- (.836,-.53);
    \draw[thick,-] (.03,1.03) -- (.896,-.47);
    \filldraw (-.866,-.5) circle (2pt) node[anchor = north east]{3};
    \draw[thick,line width = 3] (-.866,-.5) -- (.866,-.5);
    \filldraw (.866,-.5) circle (2pt) node[anchor = north west]{4};
\end{tikzpicture}	
\end{center}

\caption{The first non-reflective case with $p = 13$ is $n = 3$}
\begin{tabular}{ll}
Affine subdiagram: & $\widetilde{A}_1$\\ 
Norm 0 vector: & $v_0+3v_1+2v_2$\\
Complement: & $3v_0+13v_1$\\
Complement norm: & $52$\\
Order 2 glue: & $2v_1-3v_2$\\
\end{tabular}
\end{figure}

\pagebreak

\subsection{Tables and diagrams for $p=17$}

\begin{table}[H]
\begin{tabular}{clccc}
\hline
$\frac{k_0^2}{(e_i,e_i)}$ & $e_i$ & $(e_i,e_i)$ & $i$ & $n$\\
\hline\hline
$\frac{1}{2}$ & $v_0+3v_1+3v_2+v_3$ & $2$ & $n+2$ & $\geq 3$\\
& $v_0+4v_1+v_2+v_3+v_4$ & $2$ & $n+6$ & $\geq 4$\\\\
$\frac{16}{17}$ & $4v_0+17v_1$ & $17$ & $n+1$ & $\geq 2$\\\\
$\frac{1}{1}$ & $v_0+3v_1+3v_2$ & $1$ & $n+2$ & $2$\\
& $v_0+4v_1+v_2+v_3$ & $1$ & $n+6$ & $3$\\\\
$\frac{49}{34}$ & $7v_0+17v_1+17v_2+17v_3$ & $34$ & $n+7$ & $\geq 3$\\\\
$\frac{100}{34}$ & $10v_0+34v_1+17v_2+17v_3$ & $34$ & $n+8$ & $\geq 3$\\\\
$\frac{16}{2}$ & $4v_0+15v_1+7v_2$ & $2$ & $n+3$ & $\geq 2$\\\\
$\frac{169}{17}$ & $13v_0+51v_1+17v_2$ & $17$ & $n+4$ & $\geq 2$ \\\\
$\frac{576}{34}$ & $24v_0+85v_1+51v_2$ & $34$ & $n+5$ & $\geq 2$\\\\
$\frac{1224}{34}$ & $6v_0+22v_1+11v_2+3v_3$ & $34$ & $n+9$ & $\geq 3$\\\\

$\frac{3721}{34}$ & $61v_0+221v_1+119v_2+17v_3$ & $34$ & $n+10$ & $\geq 3$\\
\hline
\end{tabular}
\caption{Vectors found with Vinberg's Algorithm for $p = 17$}
\label{table17}
\end{table}

\begin{figure}[H]
\begin{center}
\begin{tikzpicture}[scale=.9]
	\node[draw=none, fill=none] at (0,-2.5) {$n=2$};
    \filldraw (0,2) circle (2pt) node[anchor = south]{1};
    \draw[thick,-] (.03,2.03) -- (1.594,1.277);
    \draw[thick,-] (-.03,1.97) -- (1.534,1.217);
    \filldraw (1.564,1.247) circle (2pt) node[anchor = south west]{2};
    \filldraw (1.95,-.445) circle (2pt) node[anchor = north west]{3};
    \draw[dashed,-] (1.95,-.445) -- (0,2);
    \draw[thick,line width = 3] (1.95,-.445) -- (.868,-1.802);
    \filldraw (.868,-1.802) circle (2pt) node[anchor = north west]{6};
    \draw[dashed,-] (.868,-1.802) -- (0,2);
    \draw[dashed,-] (.868,-1.802) -- (1.564,1.247);
    \filldraw (-.868,-1.802) circle (2pt) node[anchor = north east]{5};
    \draw[dashed,-] (-.868,-1.802) -- (0,2);
    \draw[dashed,-] (-.868,-1.802) -- (1.564,1.247);
    \draw[dashed,-] (-.868,-1.802) -- (1.95,-.445);
    \filldraw (-1.95,-.445) circle (2pt) node[anchor = north east]{7};
    \draw[dashed,-] (-1.95,-.445) -- (0,2);
    \draw[dashed,-] (-1.95,-.455) -- (1.564,1.247);
    \draw[dashed,-] (-1.95,-.455) -- (1.95,-.455);
    \draw[dashed,-] (-1.95,-.455) -- (.868,-1.802);
    \filldraw (-1.564,1.247) circle (2pt) node[anchor = south east]{4};
    \draw[dashed,-] (-1.564,1.247) -- (1.564,1.247);
    \draw[dashed,-] (-1.564,1.247) -- (1.95,-.455);
    \draw[dashed,-] (-1.564,1.247) -- (.868,-1.802);
    \draw[dashed,-] (-1.564,1.247) -- (-.868,-1.802);

\end{tikzpicture}

\begin{tikzpicture}[scale=1.2]
	\node[draw=none, fill=none] at (0,-2.8) {$n=3$};

    \filldraw (0,-2) circle (2pt) node[anchor = north]{4};
    \draw[thick,-] (0.03,-2.03) -- (.959,-1.801);
    \draw[thick,-] (-0.03,-1.97) -- (.899,-1.741);
    \draw[thick,-] (0.03,-2.03) -- (-.959,-1.801);
    \draw[thick,-] (-0.03,-1.97) -- (-.899,-1.741);
    \draw[thick,line width = 3] (0,-2) -- (1.646,-1.136);
    \draw[thick,line width = 3] (0,-2) -- (-1.646,-1.136);
    \draw[dashed,-] (0,-2) -- (.479,1.942);
    \draw[dashed,-] (0,-2) -- (-.479,1.942);
    
    \filldraw (.929,-1.771) circle (2pt) node[anchor = north west]{3};
    \draw[thick,-] (.959,-1.801) -- (1.676,-1.166);
    \draw[thick,-] (.899,-1.741) -- (1.616,-1.106);
    \draw[thick,line width=3] (.929,-1.771) -- (-.929,-1.771);
    \draw[dashed,-] (.929,-1.771) -- (-1.985,-.241);
    \draw[dashed,-] (.929,-1.771) -- (-1.326,1.497);
    \draw[dashed,-] (.929,-1.771) -- (1.326,1.497);
    \draw[dashed,-] (.929,-1.771) -- (1.87,.709);
    
    \filldraw (-.929,-1.771) circle (2pt) node[anchor = north east]{9};
    \draw[thick,-] (-.959,-1.801) -- (-1.676,-1.166);
    \draw[thick,-] (-.899,-1.741) -- (-1.616,-1.106);
    \draw[dashed,-] (-.929,-1.771) -- (1.985,-.241);
    \draw[dashed,-] (-.929,-1.771) -- (1.326,1.497);
    \draw[dashed,-] (-.929,-1.771) -- (-1.326,1.497);
    \draw[dashed,-] (-.929,-1.771) -- (-1.87,.709);
    
    \filldraw (1.646,-1.136) circle (2pt) node[anchor = north west ]{2};
    \draw[thick,-] (1.646,-1.136) -- (1.985,-.241);
    \draw[dashed,-] (1.646,-1.136) -- (-1.646,-1.136);
    \draw[dashed,-] (1.646,-1.136) -- (-1.985,-.241);
    \draw[dashed,-] (1.646,-1.136) -- (-1.87,.709);
    \draw[dashed,-] (1.646,-1.136) -- (-1.326,1.479);
    \draw[dashed,-] (1.646,-1.136) -- (.479,1.942);
    
    \filldraw (-1.646,-1.136) circle (2pt) node[anchor = north east]{6};
    \draw[thick,-] (-1.646,-1.136) -- (-1.985,-.241);
    \draw[dashed,-] (-1.646,-1.136) -- (1.985,-.241);
    \draw[dashed,-] (-1.646,-1.136) -- (1.87,.709);
    \draw[dashed,-] (-1.646,-1.136) -- (1.326,1.479);
    \draw[dashed,-] (-1.646,-1.136) -- (-.479,1.942);
    
    \filldraw(1.985,-.241) circle (2pt) node[anchor = west] {1};
    \draw[dashed,-] (1.985,-.241) -- (-1.985,-.241);
    \draw[dashed,-] (1.985,-.241) -- (-1.87,.709);
    \draw[dashed,-] (1.985,-.241) -- (-1.326,1.497);
    \draw[dashed,-] (1.985,-.241) -- (-.479,1.942);
    \draw[dashed,-] (1.985,-.241) -- (.479,1.942);
    \draw[dashed,-] (1.985,-.241) -- (1.87,.709);
    
    \filldraw (-1.985,-.241) circle (2pt) node[anchor = east]{12};
    \draw[dashed,-] (-1.985,-.241) -- (1.87,.709);
    \draw[dashed,-] (-1.985,-.241) -- (1.326,1.497);
    \draw[dashed,-] (-1.985,-.241) -- (.479,1.942);
    \draw[dashed,-] (-1.985,-.241) -- (-.479,1.942);
    \draw[dashed,-] (-1.985,-.241) -- (-1.87,.709);
    
    \filldraw (1.87,.709) circle (2pt) node[anchor = west]{11};
    \draw[thick,line width = 3] (1.87,.709) -- (1.326,1.497);
    \draw[dashed,-] (1.87,.709) -- (-1.87,.709);
    \draw[dashed,-] (1.87,.709)-- (-1.326,1.497);
    \draw[dashed,-] (1.87,.709)-- (-.479,1.942);
    \draw[dashed,-] (1.87,.709)-- (.479,1.942);
    
    \filldraw (-1.87,.709) circle (2pt) node[anchor = east]{8};
    \draw[thick,line width = 3] (-1.87,.709) -- (-1.326,1.497);
    \draw[dashed,-] (-1.87,.709)-- (1.326,1.497);
    \draw[dashed,-] (-1.87,.709)-- (.479,1.942);
    \draw[dashed,-] (-1.87,.709)-- (-.479,1.942);
    
    \filldraw (1.326,1.497) circle (2pt) node[anchor = south west]{10};
    \draw[dashed,-] (1.326,1.497) -- (-1.326,1.497);
    \draw[dashed,-] (1.326,1.497) -- (-.479,1.942);
    \draw[dashed,-] (1.326,1.497) -- (.479,1.942);
    
    \filldraw (-1.326,1.497) circle (2pt) node[anchor = south east]{13};
    \draw[dashed,-] (-1.326,1.497) -- (-.479,1.942);
    \draw[dashed,-] (-1.326,1.497) -- (.479,1.942);
    
    \filldraw (.479,1.942) circle (2pt) node[anchor = south west]{7};
    \draw[thick,line width = 3] (.479,1.942) -- (-.479,1.942);
    
    \filldraw (-.479,1.942) circle (2pt) node[anchor = south east]{5};

\end{tikzpicture}
\end{center}
\caption{$p=17$, reflective with $n=2 \text{ and }3$}
\end{figure}

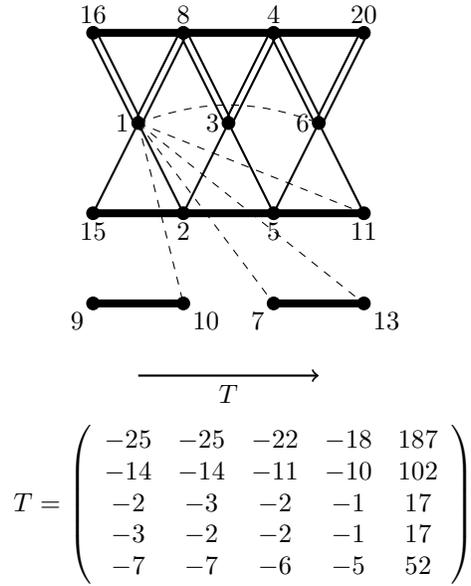
\begin{figure}[H]
\begin{tikzpicture}[scale=1.2]
    
    \filldraw (-2,1) circle (2pt) node[anchor = south]{16};
    \draw[thick,line width = 3] (-2,1) -- (-1,1);
    \filldraw (-1,1) circle (2pt) node[anchor = south]{8};
    \draw[thick,line width = 3] (-1,1) -- (0,1);
    \filldraw (0,1) circle (2pt) node[anchor = south]{4};
    \draw[thick,line width = 3] (0,1) -- (1,1);
    \filldraw(1,1) circle (2pt) node[anchor = south]{20};
    
    \filldraw (-1.5,0) circle (2pt) node[anchor = east]{1};
    \draw[thick,-] (-1.47,0.03) -- (-1.97,1.03);
    \draw[thick,-] (-1.53,-.03) -- (-2.03,.97);
    \draw[thick,-] (-1.47,-0.03) -- (-.97,.97);
    \draw[thick,-] (-1.53,.03) -- (-1.03,1.03);
    \draw[thick,-] (-1.5,0) -- (-2,-1);
    \draw[thick,-] (-1.5,0) -- (-1,-1);
    \draw[dashed,-] (-1.5,0) -- (-1,-2);
    \draw[dashed,-] (-1.5,0) -- (0,-2);
    \draw[dashed,-] (-1.5,0) -- (1,-2);
    \draw[dashed,-] (-1.5,0) -- (1,-1);
    \draw[dashed,-] (-1.5,0) to[out = 20, in = 160] (.5,0);
    \filldraw (-.5,0) circle (2pt) node[anchor = east]{3};
    \draw[thick,-] (-.47,0.03) -- (-.97,1.03);
    \draw[thick,-] (-.53,-.03) -- (-1.03,.97);
    \draw[thick,-] (-.47,-0.03) -- (.03,.97);
    \draw[thick,-] (-.53,.03) -- (-.03,1.03);
    \draw[thick,-] (-.5,0) -- (-1,-1);
    \draw[thick,-] (-.5,0) -- (0,-1);
    \filldraw (.5,0) circle (2pt) node[anchor = east]{6};
    \draw[thick,-] (.53,0.03) -- (.03,1.03);
    \draw[thick,-] (.47,-.03) -- (-.03,.97);
    \draw[thick,-] (.53,-0.03) -- (1.03,.97);
    \draw[thick,-] (.47,.03) -- (.97,1.03);
    \draw[thick,-] (.5,0) -- (0,-1);
    \draw[thick,-] (.5,0) -- (1,-1);
    
    \filldraw (-2,-1) circle (2pt) node[anchor = north]{15};
    \draw[thick,line width = 3] (-2,-1) -- (-1,-1);
    \filldraw (-1,-1) circle (2pt) node[anchor = north]{2};
    \draw[thick,line width = 3] (-1,-1) -- (0,-1);
    \filldraw (0,-1) circle (2pt) node[anchor = north]{5};
    \draw[thick,line width = 3] (0,-1) -- (1,-1);
    \filldraw(1,-1) circle (2pt) node[anchor = north]{11};

    \filldraw (-2,-2) circle (2pt) node[anchor = north east]{9};
    \draw[thick,line width = 3] (-2,-2) -- (-1,-2);
    \filldraw (-1,-2) circle (2pt) node[anchor = north west]{10};
    \filldraw (0,-2) circle (2pt) node[anchor = north east]{7};
    \draw[thick,line width = 3] (0,-2) -- (1,-2);
    \filldraw(1,-2) circle (2pt) node[anchor = north west]{13};

\draw [thick,->] (-1.5,-2.8) -- (.5,-2.8); 
	\node[draw=none, fill=none] at (-.5,-3) {$T$};

\end{tikzpicture}	
\caption{The first non-reflective case with $p = 17$ is $n = 4$.  $T$ is the transformation that shifts everything to the right.  So $1\mapsto 3$ and everything else shifts accordingly.  In order to keep the picture clean, not all dotted line edges are drawn. The transformation $T$ is given by a $5\times 5$ infinite order integer matrix.}
\begin{tabular}{ll}
\end{tabular}
\[T = \left(\begin{array}{ccccc}
-25 & -25 & -22 & -18 & 187\\
-14 & -14 & -11 & -10 & 102\\
-2 & -3 & -2 & -1 & 17\\
-3 & -2 & -2 & -1 & 17\\
-7 & -7 & -6 & -5 & 52
\end{array}\right)\]
\label{17nonref}
\end{figure}

\pagebreak

\subsection{Tables and diagrams for $p=19$}
\begin{table}[H]
\begin{tabular}{clccc}
\hline
$\frac{k_0^2}{(e_i,e_i)}$ & $e_i$ & $(e_i,e_i)$ & $i$ & $n$\\
\hline\hline
$\frac{1}{2}$ & $v_0+4v_1+2v_2+v_3$ & $2$ & $4$ & $3$\\\\
$\frac{36}{38}$ & $6v_0+19v_1+19v_2$ & $38$ & $3$ & $2$\\\\
$\frac{1}{1}$ & $v_0+4v_1+2v_2$ & $1$ & $4$ & $2$\\\\
$\frac{169}{38}$ & $13v_0+57v_1$ & $38$ & $5$ & $2$\\\\
$\frac{9}{2}$ & $3v_0+13v_1+2v_2$ & $2$ & $6$ & $2$\\
\hline
\end{tabular}

\caption{Vectors found with Vinberg's Algorithm for $p = 19$}
\begin{tabular}{l}
Norm/angle sequence for $n=2$ is $\left(2_41_238_2\right)^2$.\\
Polygon has an order 2 rotation preserving \eqref{form}.
\end{tabular}
\label{table19}
\end{table}

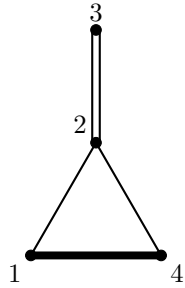
\begin{figure}[H]
\begin{center}
\begin{tikzpicture}[scale=1]

    \filldraw (0,2.5) circle (2pt) node[anchor = south]{3};
    \draw[thick,-] (-.05,2.5) -- (-.05,1);
    \draw[thick,-] (.05,2.5) -- (.05,1);
    \filldraw (0,1) circle (2pt) node[anchor = south east]{2};
    \draw[thick,-] (0,1) -- (-.866,-.5);
    \draw[thick,-] (0,1) -- (.866,-.5);
    \filldraw (-.866,-.5) circle (2pt) node[anchor = north east]{1};
    \draw[thick,line width = 3] (-.866,-.5) -- (.866,-.5);
    \filldraw (.866,-.5) circle (2pt) node[anchor = north west]{4};
\end{tikzpicture}	
\end{center}

\caption{The first non-reflective case with $p = 19$ is $n = 3$}
\begin{tabular}{ll}
Affine subdiagram: & $\widetilde{A}_1$\\ 
Norm 0 vector: & $v_0+3v_1+3v_2+v_3$\\
Complement: & $3v_0+19v_1$\\
Complement norm: & $190$\\
Glue: & none\\
\end{tabular}
\end{figure}

\pagebreak

\subsection{Tables and diagrams for $p=23$}
\label{app23}
\begin{table}[ht]

\begin{tabular}{clccc}
\hline
$\frac{k_0^2}{(e_i,e_i)}$ & $e_i$ & $(e_i,e_i)$ & $i$  & $n$ \\
\hline\hline
$\frac{1}{2}$ & $v_0+4v_1+3v_2$ & $2$ & $n+1$ & $2,3$ \\
 & $v_0+5v_1$ & $2$ & $n+2$ & $2,3$\\\\
$\frac{1}{1}$ & $ v_0+4v_1+2v_2+2v_3$ & $1$ & $6$ & $3$\\\\
$\frac{4}{2}$ & $ 2v_0+7v_1+6v_2+3v_3$ & $2$ & $7$ & $3$\\
 & $ 2v_0+9v_1+3v_2+2v_3$ & $2$ & $8$ & $3$ \\\\
$\frac{9}{2}$ & $ 3v_0+9v_1+8v_2+8v_3$ & $2$ & $9$ & $3$\\\\
$\frac{16}{1}$ & $ 4v_0+12v_1+12v_2+9v_3$ & $1$ & $10$ & $3$\\\\
$\frac{36}{2}$ & $ 6v_0+27v_1+10v_2+v_3$ & $2$ & $11$ & $3$ \\\\
$\frac{36}{1}$ & $6v_0+27v_1+10v_2$ & $1$ & $ 5$ & $2$\\\\
$\frac{49}{1}$ & $ 7v_0+32v_1+10v_2+2v_3$ & $1$ & $12$ & $3$\\\\
$\frac{100}{2}$ & $ 10v_0+33v_1+27v_2+22v_3$ & $2$ & $13$ & $3$\\\\
$\frac{144}{2}$ & $ 12v_0+55v_1+17v_2$ & $2$ & $6$ & $2$ \\ 
&&& $14$ & $3$\\\\
$\frac{225}{2}$ & $ 15v_0+48v_1+43v_2+32v_3$ & $2$ & $15$ & $3$\\\\
$\frac{400}{2}$ & $ 20v_0+92v_1+27v_2+3v_3$ & $2$ & $16$ & $3$\\\\
\hline
\end{tabular}
\caption{Vectors found with Vinberg's algorithm for $p = 23$.}

\begin{tabular}{ll}
Norm/angle sequence for $n=2$ is  $\left(2_41_24_3\right)^2$.\\
Polygon has an order 2 rotation preserving \eqref{form}. 
\end{tabular}
\end{table}

\begin{figure}[H]
\begin{center}
\begin{tikzpicture}[scale=.9]

   	\node[draw=none,fill=none,align=center] at (-4.5,.866) {19 \\ nm 1};
    \draw[thick,-] (-3.5,.866) -- (-4,0);
    \draw[thick,-] (-4,0) -- (-5,0);
    \draw[thick,-] (-5,0) -- (-5.5,.866);
    
   	\node[draw=none,fill=none,align=center] at (-3,-1.732) {21 \\ nm 2};
    \draw[thick,-] (-3.5,-.866) -- (-2.5,-.866);
    \draw[thick,-] (-2.5,-.866) -- (-2,-1.732);
    \draw[thick,-] (-2,-1.732) -- (-2.5,-.866);
    \draw[thick,-] (-2.5,-2.598) -- (-3.5,-2.598);
    \draw[thick,-] (-3.5,-2.598) -- (-4,-1.732);
    \draw[thick,-] (-4,-1.732) -- (-3.5,-.866);

   	\node[draw=none,fill=none,align=center] at (-3,0) {17 \\ nm 2};
    \draw[thick,-] (-3.5,.866) -- (-2.5,.866);
    \draw[thick,-] (-2.5,.866) -- (-2,0);
    \draw[thick,-] (-2,0) -- (-2.5,-.866);
    \draw[thick,-] (-2.5,-.866) -- (-3.5,-.866);
    \draw[thick,-] (-3.5,-.866) -- (-4,0);
    \draw[thick,-] (-4,0) -- (-3.5,.866);

   	\node[draw=none,fill=none,align=center] at (-3,1.732) {13 \\ nm 2};
    \draw[thick,-] (-2,1.732) -- (-2.5,.866);
    \draw[thick,-] (-2.5,.866) -- (-3.5,.866);
    \draw[thick,-] (-3.5,.866) -- (-4,1.732);

    \filldraw (-1,-3.464) circle (2pt);
   	\node[draw=none,fill=none,align=center] at (-1.5,-4.33) {19 \\ nm 1};
    \draw[thick,-] (-2,-3.464) -- (-1,-3.464);
    \draw[thick,-] (-1,-3.464) -- (-.5,-4.33);
    \draw[thick,-] (-2.5,-4.33) -- (-2,-3.464);
    
   	\node[draw=none,fill=none,align=center] at (-1.5,-2.598) {15 \\ nm 2};
    \draw[thick,-] (-2,-1.732) -- (-1,-1.732);
    \draw[thick,-] (-1,-1.732) -- (-.5,-2.598);
    \draw[thick,-] (-.5,-2.598) -- (-1,-3.464);
    \draw[thick,-] (-1,-3.464) -- (-2,-3.464);
    \draw[thick,-] (-2,-3.464) -- (-2.5,-2.598);
    \draw[thick,-] (-2.5,-2.596) -- (-2,-1.732);
    
   	\node[draw=none,fill=none,align=center] at (-1.5,-.866) {10 \\ nm 1};
    \draw[thick,-] (-2,0) -- (-1,0);
    \draw[thick,-] (-1,0) -- (-.5,-.866);
    \draw[thick,-] (-.5,-.866) -- (-1,0);
    \draw[thick,-] (-1,-1.732) -- (-2,-1.732);
    \draw[thick,-] (-2,-1.732) -- (-2.5,-.866);
    \draw[thick,-] (-2.5,-.866) -- (-2,0);

   	\node[draw=none,fill=none,align=center] at (-1.5,.866) {9 \\ nm 2};
    \draw[thick,-] (-2,1.732) -- (-1,1.732);
    \draw[thick,-] (-1,1.732) -- (-.5,.866);
    \draw[thick,-] (-.5,.866) -- (-1,0);
    \draw[thick,-] (-1,0) -- (-2,0);
    \draw[thick,-] (-2,0) -- (-2.5,.866);
    \draw[thick,-] (-2.5,.866) -- (-2,1.732);

   	\node[draw=none,fill=none,align=center] at (-1.5,2.598) {6 \\ nm 1};
    \draw[thick,-] (-.5,2.598) -- (-1,1.732);
    \draw[thick,-] (-1,1.732) -- (-2,1.732);
    \draw[thick,-] (-2,1.732) -- (-2.5,2.598);

   	\node[draw=none,fill=none,align=center] at (0,-3.464) {13 \\ nm 2};
    \draw[thick,-] (-.5,-2.598) -- (.5,-2.598);
    \draw[thick,-] (.5,-2.598) -- (1,-3.464);
    \draw[thick,-] (-1,-3.464) -- (-.5,-2.598);
    
   	\node[draw=none,fill=none,align=center] at (0,-1.732) {7 \\ nm 2};
    \draw[thick,-] (-.5,-.866) -- (.5,-.866);
    \draw[thick,-] (.5,-.866) -- (1,-1.732);
    \draw[thick,-] (1,-1.732) -- (.5,-2.598);
    \draw[thick,-] (.5,-2.598) -- (-.5,-2.598);
    \draw[thick,-] (-.5,-2.598) -- (-1,-1.732);
    \draw[thick,-] (-1,-1.732) -- (-.5,-.866);
    
    \filldraw (-.5,-.866) circle (2pt);
\draw(-.5,-.866) circle (4pt);
\draw(.5,.866) circle (4pt);
    \node[draw=none,fill=none,align=center] at (0,0) {1 \\ nm 2};
    \draw[thick,-] (-.5,.866) -- (.5,.866);
    \draw[thick,-] (.5,.866) -- (1,0);
    \draw[thick,-] (1,0) -- (.5,-.866);
    \draw[thick,-] (.5,-.866) -- (-.5,-.866);
    \draw[thick,-] (-.5,-.866) -- (-1,0);
    \draw[thick,-] (-1,0) -- (-.5,.866);

   	\node[draw=none,fill=none,align=center] at (0,1.732) {2 \\ nm 2};
    \draw[thick,-] (-.5,2.598) -- (.5,2.598);
    \draw[thick,-] (.5,2.598) -- (1,1.732);
    \draw[thick,-] (1,1.732) -- (.5,.866);
    \draw[thick,-] (.5,.866) -- (-.5,.866);
    \draw[thick,-] (-.5,.866) -- (-1,1.732);
    \draw[thick,-] (-1,1.732) -- (-.5,2.598);

   	\node[draw=none,fill=none,align=center] at (0,3.464) {8 \\ nm 2};
    \draw[thick,-] (1,3.464) -- (.5,2.598);
    \draw[thick,-] (.5,2.598) -- (-.5,2.598);
    \draw[thick,-] (-.5,2.598) -- (-1,3.464);

    \filldraw (2,-1.732) circle (2pt);

   	\node[draw=none,fill=none,align=left] at (1.5,-2.598) {6 \\ nm 1};
    \draw[thick,-] (1,-1.732) -- (2,-1.732);
    \draw[thick,-] (2,-1.732) -- (2.5,-2.598);
    \draw[thick,-] (.5,-2.598) -- (1,-1.732);

   	\node[draw=none,fill=none,align=left] at (1.5,-.866) {4 \\ nm 2};    
    \draw[thick,-] (1,0) -- (2,0);
    \draw[thick,-] (2,0) -- (2.5,-.866);
    \draw[thick,-] (2.5,-.866) -- (2,-1.732);
    \draw[thick,-] (2,-1.732) -- (1,-1.732);
    \draw[thick,-] (1,-1.732) -- (.5,-.866);
    \draw[thick,-] (.5,-.866) -- (1,0);
    
    \filldraw (2.5,.866) circle (2pt);
   	\node[draw=none,fill=none,align=left] at (1.5,.866) {3 \\ nm 1};
    \draw[thick,-] (1,1.732) -- (2,1.732);
    \draw[thick,-] (2,1.732) -- (2.5,.866);
    \draw[thick,-] (2.5,.866) -- (2,0);
    \draw[thick,-] (2,0) -- (1,0);
    \draw[thick,-] (1,0) -- (.5,.866);
    \draw[thick,-] (.5,.866) -- (1,1.732);

   	\node[draw=none,fill=none,align=left] at (1.5,2.598) {5 \\ nm 2};
    \draw[thick,-] (1,3.464) -- (2,3.464);
    \draw[thick,-] (2,3.464) -- (2.5,2.598);
    \draw[thick,-] (2.5,2.598) -- (2,1.732);
    \draw[thick,-] (2,1.732) -- (1,1.732);
    \draw[thick,-] (1,1.732) -- (.5,2.598);
    \draw[thick,-] (.5,2.598) -- (1,3.464);

   	\node[draw=none,fill=none,align=left] at (1.5,4.33) {12 \\ nm 1};
    \draw[thick,-] (2.5,4.33) -- (2,3.464);
    \draw[thick,-] (2,3.464) -- (1,3.464);
    \draw[thick,-] (1,3.464) -- (.5,4.33);

   	\node[draw=none,fill=none,align=left] at (3,-1.732) {8 \\ nm 2};
    \draw[thick,-] (2.5,-.866) -- (3.5,-.866);
    \draw[thick,-] (3.5,-.866) -- (4,-1.732);
    \draw[thick,-] (2,-1.732) -- (2.5,-.866);
    
   	\node[draw=none,fill=none,align=left] at (3,0) {11 \\ nm 2};
    \draw[thick,-] (2.5,.866) -- (3.5,.866);
    \draw[thick,-] (3.5,.866) -- (4,0);
    \draw[thick,-] (4,0) -- (3.5,-.866);
    \draw[thick,-] (3.5,-.866) -- (2.5,-.866);
    \draw[thick,-] (2.5,-.866) -- (2,0);
    \draw[thick,-] (2,0) -- (2.5,.866);
    
   	\node[draw=none,fill=none,align=left] at (3,1.732) {14 \\ nm 2};
    \draw[thick,-] (2.5,2.598) -- (3.5,2.598);
    \draw[thick,-] (3.5,2.598) -- (4,1.732);
    \draw[thick,-] (4,1.732) -- (3.5,.866);
    \draw[thick,-] (3.5,.866) -- (2.5,.866);
    \draw[thick,-] (2.5,.866) -- (2,1.732);
    \draw[thick,-] (2,1.732) -- (2.5,2.598);

   	\node[draw=none,fill=none,align=left] at (3,3.464) {16 \\ nm 2};
    \draw[thick,-] (2.5,4.33) -- (3.5,4.33);
    \draw[thick,-] (3.5,4.33) -- (4,3.464);
    \draw[thick,-] (4,3.464) -- (3.5,2.598);
    \draw[thick,-] (3.5,2.598) -- (2.5,2.598);
    \draw[thick,-] (2.5,2.598) -- (2,3.464);
    \draw[thick,-] (2,3.464) -- (2.5,4.33);

   	\node[draw=none,fill=none,align=left] at (3,5.196) {22 \\ nm 2};
    \draw[thick,-] (4,5.196) -- (3.5,4.33);
    \draw[thick,-] (3.5,4.33) -- (2.5,4.33);
    \draw[thick,-] (2.5,4.33) -- (2,5.196);

\filldraw (5,0) circle (2pt);
   	\node[draw=none,fill=none,align=left] at (4.5,-.866) {12 \\ nm 1};
    \draw[thick,-] (4,0) -- (5,0);
    \draw[thick,-] (5,0) -- (5.5,-.866);
    \draw[thick,-] (3.5,-.866) -- (4,0);
    
   	\node[draw=none,fill=none,align=left] at (4.5,.866) {18 \\ nm 2};
    \draw[thick,-] (4,1.732) -- (5,1.732);
    \draw[thick,-] (5,1.732) -- (5.5,.866);
    \draw[thick,-] (5.5,.866) -- (5,0);
    \draw[thick,-] (5,0) -- (4,0);
    \draw[thick,-] (4,0) -- (3.5,.866);
    \draw[thick,-] (3.5,.866) -- (4,1.732);
    
   	\node[draw=none,fill=none,align=left] at (4.5,2.598) {20 \\ nm 1};
    \draw[thick,-] (4,3.464) -- (5,3.464);
    \draw[thick,-] (5,3.464) -- (5.5,2.598);
    \draw[thick,-] (5.5,2.598) -- (5,1.732);
    \draw[thick,-] (5,1.732) -- (4,1.732);
    \draw[thick,-] (4,1.732) -- (3.5,2.598);
    \draw[thick,-] (3.5,2.598) -- (4,3.464);

   	\node[draw=none,fill=none,align=left] at (6,0) {22 \\ nm 2};
    \draw[thick,-] (5.5,.866) -- (6.5,.866);
    \draw[thick,-] (6.5,.866) -- (7,0);
    \draw[thick,-] (5.5,.866) -- (6.5,.866);

\end{tikzpicture}
\end{center}
\caption{The fundamental polyhedron for $p = 23$, $n=3$.  The open faces on the top match up with the with the ones with the same labels on the bottom to make a cylinder.  The closed dots are a $G$ orbit and the open dots are an $R$ orbit.}
\label{23poly}
\end{figure}
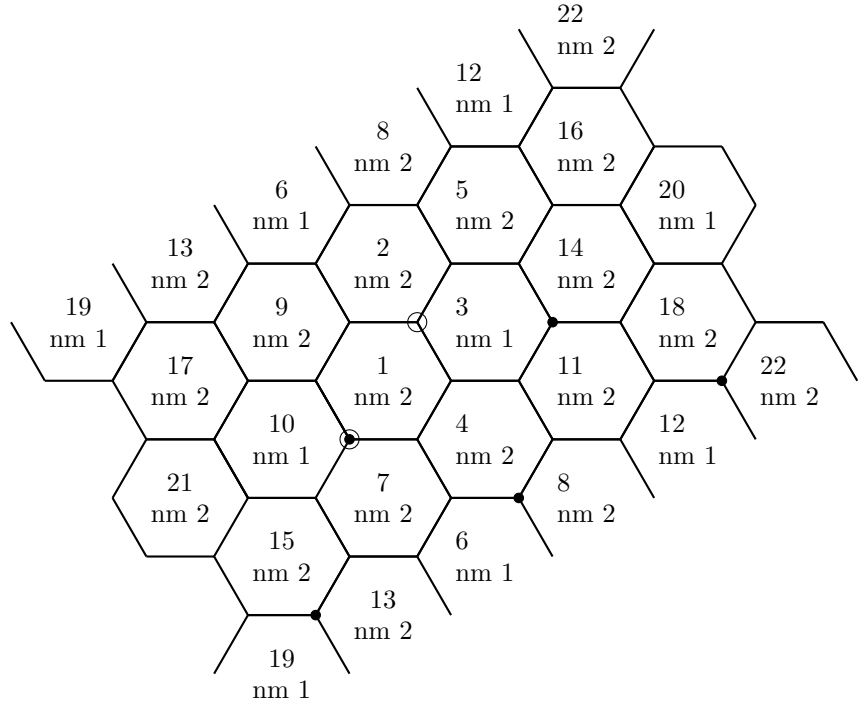

\[G = \left(\begin{array}{cccc}
-46 & -37 & -32 & 322\\
-13 & -10 & -10 & 92\\
-4 & -2 & -2 & 23\\
-10 & -8 & -7 & 70\\
\end{array}\right)
R = \left(\begin{array}{cccc}
-18 & -19 & -12 & 138\\
-19 & -18 & -12 & 138\\
-12 & -12 & -9 & 92\\
-6 & -6 & -4 & 45
\end{array}\right)\]

\bibliographystyle{srtnumbered}

\bibliography{mybib}

\subsection*{Acknowledgment}
Thanks to Daniel Allcock for posing the question.

\end{document}